\documentclass[12pt]{amsart}

\usepackage[utf8]{inputenc}
\usepackage{bm}
\usepackage[dvipsnames,svgnames,x11names]{xcolor}
\usepackage[hyphens]{url}
\usepackage[
    colorlinks=true,
    linkcolor=Maroon,
    citecolor=blue,
    urlcolor=blue,
    hypertexnames=false,
    linktocpage
]{hyperref}
\usepackage{mathrsfs}
\usepackage{bookmark}
\usepackage{amsmath,thmtools,mathtools}
\usepackage{amssymb}
\mathtoolsset{showonlyrefs=true}
\usepackage{graphicx}
\usepackage{tikz}
\usepackage{tikz-3dplot}
\usetikzlibrary{arrows.meta,calc,intersections}
\usepackage{fancyhdr}
\usepackage{comment}
\usepackage{esint}
\usepackage{enumerate}
\usepackage{caption}
\usepackage{needspace}

\allowdisplaybreaks

\declaretheorem[name=Theorem,numberwithin=section]{thm}
\declaretheorem[name=Remark,style=remark,sibling=thm]{rem}
\declaretheorem[name=Lemma,sibling=thm]{lemma}
\declaretheorem[name=Proposition,sibling=thm]{prop}

\declaretheorem[name=Question,numbered=no]{question}

\numberwithin{equation}{section}

\newcommand{\bbN}{\mathbb{N}}

\newcommand{\bbR}{\mathbb{R}}

\newcommand{\bbS}{\mathbb{S}}

\newcommand{\bbB}{\mathbb{B}}

\newcommand{\Sn}{\mathbb{S}^n}

\newcommand{\cF}{\mathcal{F}}

\newcommand{\cL}{\mathcal{L}}
\newcommand{\cM}{\mathcal{M}}
\newcommand{\cN}{\mathcal{N}}

\DeclareMathOperator{\supp}{supp}

\newcounter{pasosdemo}

\begin{document}

\title[On a question by Firey]
{On a question by Firey concerning uniqueness}

\author[]{Carlos Cabezas-Moreno}

\begin{abstract}
We prove that the curvature equation
\[
\sum_{j=1}^n
\alpha_jE_j(\tau_{\cM})
=
\sum_{j=1}^n
\alpha_jE_j(\tau_{\cN})
\]
yields uniqueness up to translation for any two closed $C^2_+$
hypersurfaces
$\cM,\,\cN\hookrightarrow\bbR^{n+1}$ whenever
$(\alpha_1,\ldots,\alpha_n)\in\bbR_{\geq0}^n\setminus\{0\}$
is log-concave and has no internal zeros.
This gives an affirmative answer to a uniqueness question posed
by Firey in the $C^2_+$ class.
\end{abstract}

\maketitle

\section{Introduction}

Let $(\bbR^{n+1},\delta:=\langle\cdot,\cdot\rangle,\bar D)$ be the
standard Euclidean space equipped with its flat metric and connection,
and assume $n\geq2$ throughout the paper. Denote by
$(\bbS^n,\bar g,\bar\nabla)$ the unit sphere with its induced
round metric and Levi-Civita connection, and by $d\sigma$ its
Riemannian volume form.

Throughout the paper, a convex body is a compact convex subset
of $\bbR^{n+1}$ with non-empty interior.
It is of class $C^2_+$ if its boundary is a $C^2$ embedded
hypersurface with positive principal curvatures.

For a closed convex hypersurface
$\cM\hookrightarrow\bbR^{n+1}$ of class $C^2_+$, let
$X_{\cM}\colon\Sn\rightarrow\cM$ denote the inverse of the
Gauss map defined by the outward unit normal.
The support function of $\cM$ is given by
\begin{equation}
h_{\cM}(u)
=
\langle X_{\cM}(u),u\rangle,
\end{equation}
and the Gauss parametrization satisfies
\begin{equation}
X_{\cM}(u)
=
h_{\cM}(u)u+\bar\nabla h_{\cM}(u).
\end{equation}

For $f\in C^2(\Sn)$, define the self-adjoint endomorphism
$\tau[f]$ of $T\Sn$ by
\begin{equation}
\tau[f]^i_j
=
\bar g^{ik}
\left(
\bar\nabla_{kj}^2f+f\bar g_{kj}
\right).
\end{equation}
When $f=h_{\cM}$, the endomorphism $\tau[h_{\cM}]$ is positive
definite, and its eigenvalues are the principal radii of curvature
of $\cM$.
We write
\[
\tau_{\cM}:=\tau[h_{\cM}],
\qquad
\tau[f]_{ij}:=\bar\nabla_{ij}^2f+f\bar g_{ij}.
\]
Consider the curvature functional
\begin{equation}
\cF_\alpha(\cdot)
:=
\sum_{j=1}^n
\alpha_jE_j(\cdot),
\end{equation}
where, for $j\in\{0,\ldots,n\}$, $E_j$ is the $j$-th elementary
symmetric polynomial, normalized so that $E_j(1,\ldots,1)=1$.
Both $\cF_\alpha$ and the elementary symmetric polynomials $E_j$
are understood as functions of the eigenvalues of a self-adjoint
endomorphism of the tangent bundle of $\bbS^n$.

In this note we study a uniqueness problem posed by Firey
at the 1974 International Congress of Mathematicians in Vancouver
\cite{FireyICM74}, related to the kinematic measures for sets of
support figures defined in \cite{Firey74}.
The problem was subsequently studied by Schneider
\cite{Schneider75,Schneider76}. For $n=2$, uniqueness up to
translation for sufficiently smooth convex bodies follows from a
theorem of A.~D.~Alexandrov; see \cite[Sec.~9]{Firey74}.
In this dimension, the corresponding uniqueness result for general
measures was proved by Schneider \cite{Schneider76}.

Consider convex bodies $K$ and $L$, and denote their $i$-th area
measures by $S_i(K,\cdot)$ and $S_i(L,\cdot)$, respectively.
By \cite[Thm.~4.4.6]{Schneider14}, for all Borel sets
$\omega,\eta\subseteq\bbS^n$,
\begin{equation}
\begin{split}
&\int_{SO(n+1)}
S_k(K+\mathfrak{g}L,\omega\cap\mathfrak{g}\eta)
\,d\nu(\mathfrak{g})
\\
&\qquad=
\frac{1}{|\bbS^n|}
\sum_{i=0}^k
\binom{k}{i}
S_{k-i}(K,\omega)S_i(L,\eta),
\end{split}
\end{equation}
where $k\in\{1,\ldots,n\}$ and $\nu$ is the normalized Haar measure
on $SO(n+1)$.
In particular, taking $\eta=\bbS^n$, we obtain
\begin{equation}
\begin{split}
\mu_k(K,L,\omega)
&:=
\frac{|\bbS^n|}{n+1}
\int_{SO(n+1)}
S_k(K+\mathfrak{g}L,\omega)
\,d\nu(\mathfrak{g})
\\
&=
\sum_{i=0}^k
\binom{k}{i}
S_{k-i}(K,\omega)W_{n+1-i}(L),
\end{split}
\end{equation}
where $W_{n+1-i}(L)$ denotes the $(n+1-i)$-th quermassintegral
of $L$.

If $\cM=\partial K$ is of class $C^2_+$, the density of
$\mu_n(K,L,\cdot)$ with respect to $d\sigma$ is
\begin{equation}
\frac{d\mu_n(K,L,\cdot)}{d\sigma}
=
\sum_{j=0}^n
\binom{n}{j}W_{j+1}(L)E_j(\tau_{\cM}).
\label{eq:kinem:density}
\end{equation}

Firey asked the following:
\begin{question}
Given $L$, does the measure $\mu_n(K,L,\cdot)$ determine
$\partial K$ up to a translation?
\end{question}

This uniqueness problem can be generalized in the $C^2_+$ class
as follows. Consider a coefficient vector
\[
\alpha=(\alpha_1,\ldots,\alpha_n)
\in\bbR_{\geq0}^n\setminus\{0\}.
\]
Does the identity
\begin{equation}
\cF_\alpha(\tau_{\cM})
=
\cF_\alpha(\tau_{\cN})
\label{eq:Firey-equation}
\end{equation}
force the $C^2_+$ hypersurfaces $\cM$ and $\cN$ to differ by a
translation?
A partial result was obtained by Ivaki \cite{Ivaki23}, who answered
Firey's question in the even isotropic case.
In the $C^2_+$ setting, \eqref{eq:Firey-equation} is the density
formulation of the uniqueness problem for a linear combination
of area measures. It generalizes the setting of the
Alexandrov--Fenchel--Jessen theorem; see \cite{Chern} and
\cite[p.~454]{Schneider14}.

Here we provide an affirmative answer to Firey's question in the
$C^2_+$ class. Indeed, if
\[
\mu_n(K_1,L,\cdot)
=
\mu_n(K_2,L,\cdot)
\]
for a fixed $L$, then, after cancelling the $j=0$ term in
\eqref{eq:kinem:density}, this corresponds to choosing
\[
\alpha_j
=
\binom{n}{j}W_{j+1}(L),
\qquad
1\leq j\leq n.
\]
By the Alexandrov--Fenchel inequality and the log-concavity of the
binomial coefficients, this coefficient sequence is log-concave.
Moreover, all its entries are positive, so it has no internal zeros.

We prove uniqueness under the weaker assumptions that
$(\alpha_1,\ldots,\alpha_n)$ is log-concave, namely
\begin{equation}
\alpha_j^2
\geq
\alpha_{j-1}\alpha_{j+1},
\qquad
2\leq j\leq n-1,
\end{equation}
and has no internal zeros, meaning that
\[
\alpha_i\alpha_k>0,\quad i<j<k
\quad\Longrightarrow\quad
\alpha_j>0.
\]
Equivalently, the set $\{j:\alpha_j>0\}$ is an interval of integers.

Our main result is the following.

\begin{thm}
\label{thm:main}
Let $\cM,\,\cN\hookrightarrow\bbR^{n+1}$ be closed convex
hypersurfaces of class $C^2_+$.
Suppose that $\alpha_1,\ldots,\alpha_n\geq0$ are not all zero
and form a log-concave sequence with no internal zeros.
If
\begin{equation}
\sum_{j=1}^n
\alpha_jE_j(\tau_{\cM})
=
\sum_{j=1}^n
\alpha_jE_j(\tau_{\cN}),
\label{eq:main-equation}
\end{equation}
then there exists $v\in\bbR^{n+1}$ such that
\[
\cM=\cN+v.
\]
\end{thm}

The algebraic ingredient of the proof is the following homogeneous polynomial
\begin{equation}
s_\alpha(x,y,z)
=
\sum_{j=1}^n
\alpha_j
\left(
\sum_{r=0}^{j-1}x^ry^{j-1-r}
\right)z^{n-j},
\end{equation}
which arises from a factorization of
\eqref{eq:main-equation}.
We prove that the polynomial above is dually Lorentzian if and only if
$(\alpha_1,\ldots,\alpha_n)$ is log-concave and has no internal zeros.

Since $\deg s_\alpha=n-1$, the generalized Alexandrov--Fenchel
inequality of Ross--S\"u{\ss}--Wannerer \cite{RSW} applies in
$\bbR^{n+1}$. The curvature equation forces equality in this
inequality, and the characterization of the equality cases yields the desired
rigidity.

In Section~\ref{sec:spectral-linearization}, we also establish a
weighted spectral inequality, including its equality case, and
show that the kernel of the linearized curvature operator consists
precisely of the restrictions of linear functions to the unit sphere
$\Sn$.
Together with ellipticity and formal self-adjointness, this
description of the kernel ensures invertibility of the linearized
operator on the $L^2(d\sigma)$-orthogonal complement of this
subspace, in suitable H\"older spaces.
This provides the openness step for a continuity approach to the existence problem for the equation
\[
\cF_\alpha(\tau[h])=g,
\]
after fixing translations and requiring the prescribed function $g$
to be $L^2(d\sigma)$-orthogonal to the linear functions.

\section{Dually Lorentzian polynomials and the generalized
Alexandrov--Fenchel inequality}

Throughout this section we follow the notation and terminology
of \cite{BrandenHuh} and \cite{RSW}.

\subsection{Mixed discriminants and mixed volumes}

For self-adjoint endomorphisms $H_1,\ldots,H_n$ of an
$n$-dimensional Euclidean space, their mixed discriminant is
\begin{equation}
D(H_1,\ldots,H_n)
:=
\frac{1}{n!}
\left.
\frac{\partial^n}{\partial t_1\cdots\partial t_n}
\det\left(t_1H_1+\cdots+t_nH_n\right)
\right|_{t=0}.
\end{equation}

It is symmetric and multilinear, and satisfies
\[
D(H,\ldots,H)=\det H.
\]

We write $H[k]$ when $H$ is repeated $k$ times.
Let $I$ denote the identity endomorphism.
Since
\[
\begin{aligned}
\det(I+tH)
&=
\sum_{j=0}^n
\binom{n}{j}
D(H[j],I[n-j])t^j
\\
&=
\sum_{j=0}^n
\binom{n}{j}
E_j(H)t^j,
\end{aligned}
\]
comparison of coefficients gives
\begin{equation}
E_j(H)
=
D(H[j],I[n-j]).
\label{eq:E-D}
\end{equation}
If $K$ denotes the convex body bounded by $\cM$, then its
$j$-th area measure satisfies
\begin{equation}
dS_j(K)
=
D(\tau_{\cM}[j],I[n-j])\,d\sigma,
\end{equation}
where $j\in\{0,\dots,n\}$.
Moreover, if $K_0,\ldots,K_n\subset\bbR^{n+1}$ are convex bodies
of class $C^2_+$ with support functions $h_0,\ldots,h_n$, then
their mixed volume is given by
\begin{equation}
V(K_0,\ldots,K_n)
=
\frac{1}{n+1}
\int_{\Sn}
h_0D(\tau[h_1],\ldots,\tau[h_n])
\,d\sigma.
\label{eq:mixed-volume-formula}
\end{equation}
\subsection{Dually Lorentzian polynomials}

We recall here the definition of a Lorentzian polynomial from
\cite[Def.~2.1]{RSW}; see also \cite{BrandenHuh}.
Here $\bbN$ includes zero, and $e_1,\ldots,e_q$ denote the
standard coordinate vectors.

A set $J\subseteq\bbN^q$ is called $M$-convex if it satisfies
the following exchange property: whenever $\mu,\nu\in J$
and $\mu_i>\nu_i$, there exists an index $j$ such that
\[
\mu_j<\nu_j
\qquad\text{and}\qquad
\mu-e_i+e_j\in J.
\]

Let
\[
p(x)=\sum_{|\mu|=d}c_\mu x^\mu,
\qquad c_\mu\geq0,
\]
be homogeneous of degree $d\geq2$, where
$|\mu|:=\mu_1+\cdots+\mu_q$ and
$x^\mu:=x_1^{\mu_1}\cdots x_q^{\mu_q}$.
Its support is
\[
\supp p:=\{\mu\in\bbN^q:c_\mu>0\}.
\]
We call $p$ Lorentzian if $\supp p$ is $M$-convex and,
for every multi-index $\nu$ with $|\nu|=d-2$, the Hessian
of the quadratic polynomial
\[
\partial^\nu p
:=
\frac{\partial^{|\nu|}p}
{\partial x_1^{\nu_1}\cdots\partial x_q^{\nu_q}}
\]
has at most one positive eigenvalue, counted with multiplicity.
For degrees zero and one, every homogeneous polynomial
with non-negative coefficients is declared Lorentzian.

For a multi-index $\mu=(\mu_1,\ldots,\mu_q)\in\bbN^q$, write
\[
\mu!=\mu_1!\cdots\mu_q!.
\]
Define the normalization operator by
\begin{equation}
N(x^\mu)=\frac{x^\mu}{\mu!},
\end{equation}
extended to polynomials by linearity.
For a polynomial $s=s(x_1,\ldots,x_q)$, we call
$\kappa\in\bbN^q$ admissible if
$\deg_{x_i}s\leq\kappa_i$ for every $i$.
For such an admissible multi-index $\kappa$, define the dual of $s$ by
\begin{equation}
s^\vee(x)
=
N\left(
x^\kappa s(x_1^{-1},\ldots,x_q^{-1})
\right).
\end{equation}
The polynomial $s$ is dually Lorentzian if $s^\vee$ is Lorentzian.
This property is independent of the admissible choice of $\kappa$;
see \cite[Rem.~4.3]{RSW}.

We shall use the following characterization from
\cite[Ex.~2.26]{BrandenHuh}.
\begin{prop}
\label{prop:characterization-Lorentzian}
Let
\[
p(u,z)=\sum_{k=0}^m c_ku^kz^{m-k},
\qquad c_k\geq0,
\]
be non-zero.
Then $p$ is Lorentzian if and only if
$(c_0,\ldots,c_m)$ has no internal zeros and
\begin{equation}
\left(\frac{c_k}{\binom{m}{k}}\right)^2
\geq
\frac{c_{k-1}}{\binom{m}{k-1}}
\frac{c_{k+1}}{\binom{m}{k+1}},
\qquad
1\leq k\leq m-1.
\end{equation}
These inequalities express the ultra-log-concavity of
$(c_0,\ldots,c_m)$.
\end{prop}

\begin{rem}
Lorentzianity is preserved under linear substitutions with
non-negative coefficients, coordinatewise multidegree truncations,
and formal antiderivatives. The latter are defined on monomials by
\begin{equation}
\int x^\mu\,dx^\nu
=
\frac{\mu!}{(\mu+\nu)!}x^{\mu+\nu}.
\end{equation}
These preservation properties follow from
\cite[Thm.~3.1(1), Prop.~3.3 and Cor.~5.9]{RSW}.
Linear substitutions with non-negative coefficients also preserve
dual Lorentzianity \cite[Thm.~5.12]{RSW}.
\end{rem}

\subsection{The generalized Alexandrov--Fenchel inequality}

For a homogeneous polynomial of degree $d-2$,
\[
s(x_1,\ldots,x_m)=\sum_\gamma c_\gamma x^\gamma,
\]
we write
\begin{equation}
\begin{split}
V(K,L,s(C_1,\ldots,C_m))
:=
\sum_\gamma c_\gamma
V(K,L,C_1[\gamma_1],\ldots,C_m[\gamma_m]).
\end{split}
\end{equation}

We shall use the following theorem of Ross--S\"u{\ss}--Wannerer
\cite[Thm.~8.8]{RSW}.

\begin{thm}
\label{thm:RSW}
Let $s$ be a non-zero dually Lorentzian polynomial of degree
$d-2$, and let $K,L,C_1,\ldots,C_m\subset\bbR^d$ be convex bodies.
Then
\begin{equation}
\begin{split}
V(K,L,s(C_1,\ldots,C_m))^2
\geq{}&
V(K,K,s(C_1,\ldots,C_m))
V(L,L,s(C_1,\ldots,C_m)).
\end{split}
\end{equation}
If all the bodies involved are of class $C^2_+$, equality holds
if and only if $K$ and $L$ are homothetic.
\end{thm}

\subsection{The relevant polynomial}

Set $w:=h_{\cM}-h_{\cN}$.
By linearity of the curvature tensor,
\begin{equation}
\tau_{\cM}-\tau_{\cN}
=
\tau[w].
\label{eq:tau-difference}
\end{equation}
Using \eqref{eq:E-D}, equation \eqref{eq:main-equation} becomes
\begin{equation}
\sum_{j=1}^n
\alpha_jD(\tau_{\cM}[j],I[n-j])
=
\sum_{j=1}^n
\alpha_jD(\tau_{\cN}[j],I[n-j]).
\end{equation}
For every $1\leq j\leq n$, multilinearity and symmetry give
\begin{equation}
\begin{split}
&D(\tau_{\cM}[j],I[n-j])
-
D(\tau_{\cN}[j],I[n-j])
\\
&\qquad=
\sum_{r=0}^{j-1}
D(
\tau_{\cM}-\tau_{\cN},
\tau_{\cM}[r],
\tau_{\cN}[j-1-r],
I[n-j]).
\end{split}
\label{eq:telescoping}
\end{equation}
This is the mixed-discriminant analogue of
\[
x^j-y^j
=
(x-y)\sum_{r=0}^{j-1}x^ry^{j-1-r}.
\]

Accordingly, define
\begin{equation}
s_\alpha(x,y,z)
:=
\sum_{j=1}^n
\alpha_j
\left(
\sum_{r=0}^{j-1}x^ry^{j-1-r}
\right)z^{n-j}.
\end{equation}
This polynomial is homogeneous of degree $n-1$ and satisfies
\[
(x-y)s_\alpha(x,y,z)
=
\sum_{j=1}^n
\alpha_j(x^j-y^j)z^{n-j}.
\]
Combining \eqref{eq:tau-difference} and \eqref{eq:telescoping},
we obtain the exact factorization
\begin{equation}
\sum_{j=1}^n
\alpha_j
\sum_{r=0}^{j-1}
D(
\tau[w],
\tau_{\cM}[r],
\tau_{\cN}[j-1-r],
I[n-j])
=
0
\label{eq:exact-factorization}
\end{equation}
pointwise on $\Sn$.

\section{Log-concavity and dual Lorentzianity}

The following characterization is the algebraic ingredient of
the proof.

\begin{prop}
\label{prop:secant-DL}
Let $\alpha_1,\ldots,\alpha_n\geq0$ be not all zero.
Then $s_\alpha$ is dually Lorentzian if and only if
$(\alpha_1,\ldots,\alpha_n)$ is log-concave and has no internal zeros.
\end{prop}

\begin{proof}
Set $m:=n-1$ and $a_k:=\alpha_{k+1}$ for $0\leq k\leq m$.
Then
\begin{equation}
s_\alpha(x,y,z)
=
\sum_{k=0}^m
a_k
\left(
\sum_{\ell=0}^k x^\ell y^{k-\ell}
\right)z^{m-k}.
\label{eq:s-expanded}
\end{equation}
Consider the bivariate polynomial
\[
p(u,z)=\sum_{k=0}^m a_ku^kz^{m-k}.
\]
Its dual with respect to the admissible multidegree $(m,m)$ is
\begin{equation}
p^\vee(u,z)
=
\sum_{k=0}^m
\frac{a_k}{(m-k)!k!}
u^{m-k}z^k.
\label{eq:p-dual}
\end{equation}
The coefficient of $u^rz^{m-r}$, divided by $\binom mr$, is
$a_{m-r}/m!$.
Thus Proposition~\ref{prop:characterization-Lorentzian} shows that
$p^\vee$ is Lorentzian if and only if
$(a_0,\ldots,a_m)$ is log-concave and has no internal zeros.

Assume first that these conditions hold.
Taking the formal antiderivative of $p^\vee$ with multi-index
$(m,0)$, we obtain
\[
q(u,z)
:=
\int p^\vee(u,z)\,du^m
=
\sum_{k=0}^m
\frac{a_k}{(2m-k)!k!}
u^{2m-k}z^k.
\]
By \cite[Cor.~5.9]{RSW}, $q$ is Lorentzian.
The substitution $u=x+y$ preserves Lorentzianity by
\cite[Thm.~3.1(1)]{RSW}, so
\begin{equation}
\begin{split}
Q(x,y,z)
&:=
q(x+y,z)
\\
&=
\sum_{k=0}^m
\sum_{r=0}^{2m-k}
\frac{a_k}{k!r!(2m-k-r)!}
x^ry^{2m-k-r}z^k
\end{split}
\label{eq:Q-expanded}
\end{equation}
is Lorentzian.

Truncate $Q$ coordinatewise to multidegree $(m,m,m)$.
By \cite[Prop.~3.3]{RSW}, the resulting polynomial remains
Lorentzian.
For fixed $k$, a term in \eqref{eq:Q-expanded} survives precisely
when
\[
r\leq m,
\qquad
2m-k-r\leq m,
\]
or equivalently $m-k\leq r\leq m$.
We then write $r=m-\ell$, so that
\begin{equation}
Q_{\leq(m,m,m)}
=
\sum_{k=0}^m
\sum_{\ell=0}^k
\frac{a_k\,x^{m-\ell}y^{m-k+\ell}z^k}
{k!(m-\ell)!(m-k+\ell)!}.
\end{equation}

The multidegree $(m,m,m)$ is admissible for $s_\alpha$.
The dual of \eqref{eq:s-expanded} with respect to this
multidegree is
\begin{equation}
s_\alpha^\vee(x,y,z)
=
\sum_{k=0}^m
\sum_{\ell=0}^k
\frac{a_k\,x^{m-\ell}y^{m-k+\ell}z^k}
{k!(m-\ell)!(m-k+\ell)!}.
\end{equation}
Hence $s_\alpha^\vee=Q_{\leq(m,m,m)}$ is Lorentzian, and
$s_\alpha$ is dually Lorentzian.

Conversely, suppose that $s_\alpha$ is dually Lorentzian.
By \cite[Thm.~5.12]{RSW}, setting $y=0$ preserves dual
Lorentzianity. Therefore
\[
s_\alpha(x,0,z)
=
\sum_{k=0}^m a_kx^kz^{m-k}
=
p(x,z)
\]
is dually Lorentzian.
By \cite[Rem.~4.3]{RSW}, its dual with respect to $(m,m)$,
given in \eqref{eq:p-dual}, is Lorentzian.
It follows from the characterization above that
$(a_0,\ldots,a_m)$, and hence
$(\alpha_1,\ldots,\alpha_n)$, is log-concave and has no
internal zeros.
\end{proof}

\section{Proof of the main theorem}

\begin{proof}[Proof of Theorem~\ref{thm:main}]
Let $K$ and $L$ be the convex bodies bounded by $\cM$ and $\cN$,
respectively, and let $\bbB$ denote the Euclidean unit ball of
$\bbR^{n+1}$.

Multiplying \eqref{eq:exact-factorization} by $h_{\cM}$,
integrating over $\Sn$, and using
\eqref{eq:mixed-volume-formula}, we obtain
\[
V(K,K,s_\alpha(K,L,\bbB))
=
V(K,L,s_\alpha(K,L,\bbB)).
\]
Multiplying instead by $h_{\cN}$, we find
\[
V(L,K,s_\alpha(K,L,\bbB))
=
V(L,L,s_\alpha(K,L,\bbB)).
\]
By symmetry of mixed volumes,
\begin{equation}
\begin{split}
V(K,K,s_\alpha(K,L,\bbB))
&=
V(K,L,s_\alpha(K,L,\bbB))
\\
&=
V(L,L,s_\alpha(K,L,\bbB)).
\end{split}
\label{eq:equal-mixed-volumes}
\end{equation}
Consequently,
\begin{equation}
\begin{split}
V(K,L,s_\alpha(K,L,\bbB))^2
=
V(K,K,s_\alpha(K,L,\bbB))
V(L,L,s_\alpha(K,L,\bbB)).
\end{split}
\label{eq:AF-equality}
\end{equation}

By Proposition~\ref{prop:secant-DL}, $s_\alpha$ is a non-zero
dually Lorentzian polynomial of degree
\[
n-1=(n+1)-2.
\]
Thus \eqref{eq:AF-equality} is the equality case of
Theorem~\ref{thm:RSW} in $\bbR^{n+1}$.
Since $\cM$, $\cN$, and $\bbS^n$ are of class $C^2_+$, there exist
$\lambda>0$ and $v\in\bbR^{n+1}$ such that
\[
\cM=\lambda\cN+v.
\]

Keeping the arguments in $s_\alpha(K,L,\bbB)$ fixed, translation
invariance and homogeneity in the first argument give
\[
V(K,L,s_\alpha(K,L,\bbB))
=
\lambda V(L,L,s_\alpha(K,L,\bbB)).
\]
By \eqref{eq:equal-mixed-volumes}, the two mixed volumes in this
identity coincide.
Moreover,
\[
V(L,L,s_\alpha(K,L,\bbB))>0,
\]
since the coefficients of $s_\alpha$ are non-negative and not all
zero, and all the bodies have non-empty interior.
Therefore $\lambda=1$, and hence $\cM=\cN+v$.
\end{proof}

\section{Spectral inequality and linearization}
\label{sec:spectral-linearization}

Throughout this section, we assume that
$\alpha_1,\ldots,\alpha_n\geq0$ are not all zero and form a
log-concave sequence with no internal zeros.
All hypersurfaces are closed and smooth with positive principal
curvatures. In particular, their support functions satisfy
\[
\tau[h]>0.
\]
All functions considered below are smooth.

The proof of Theorem~\ref{thm:main} uses the global equality case
of the generalized Alexandrov--Fenchel inequality.
We now derive a weighted spectral inequality and determine the
kernel of the linearized curvature operator.

\subsection{The linearized operator}

Let $h$ be the support function of a smooth convex hypersurface
$\cM$ satisfying $\tau[h]>0$, and let $\psi\in C^\infty(\Sn)$.
There exists $\varepsilon>0$ such that, for every
$t\in(-\varepsilon,\varepsilon)$, the function $h+t\psi$ is the
support function of a smooth convex hypersurface $\cM_t$ with
$\tau[h+t\psi]>0$.
Define
\begin{equation}
\cL_\alpha(h)[\psi]
:=
\left.\frac{d}{dt}\right|_{t=0}
\cF_\alpha(\tau[h+t\psi]).
\end{equation}
Thus $\cL_\alpha(h)$ is the linearization at $h$ of the nonlinear
map $h\mapsto\cF_\alpha(\tau[h])$ on smooth support functions.

For every self-adjoint endomorphism $H$ of $T\Sn$, define
$\cF_\alpha^{ij}(\tau[h])$ by
\begin{equation}
\begin{split}
\left.\frac{d}{dt}\right|_{t=0}
\cF_\alpha(\tau[h]+tH)
&=
\cF_\alpha^{ij}(\tau[h])H_{ij}
\\
&=
\sum_{k=1}^n
k\alpha_k
D(H,\tau[h][k-1],I[n-k]).
\end{split}
\label{eq:Falpha-first-derivative}
\end{equation}
Thus $\cF_\alpha^{ij}(\tau[h])$ is the symmetric tensor
representing the differential of $\cF_\alpha$ at $\tau[h]$.
Given $\lambda=(\lambda_1,\ldots,\lambda_n)\in\bbR^n$, let
$E_{j;i}(\lambda)$ denote the normalized $j$-th elementary
symmetric polynomial in the $n-1$ variables obtained by omitting
$\lambda_i$. For $1\leq j\leq n-1$,
\begin{equation}
E_{j;i}(\lambda)
=
\binom{n-1}{j}^{-1}
\sum_{\substack{
1\leq i_1<\cdots<i_j\leq n\\
i_1,\ldots,i_j\neq i
}}
\lambda_{i_1}\cdots\lambda_{i_j},
\end{equation}
and we set $E_{0;i}=1$.
The standard identities for elementary symmetric polynomials,
with this normalization, give
\begin{equation}
\frac{\partial E_j}{\partial\lambda_i}
=
\frac{j}{n}E_{j-1;i},
\qquad
1\leq j\leq n,
\label{eq:Ej-derivative}
\end{equation}
and
\begin{equation}
\sum_{i=1}^n E_{j;i}
=
nE_j,
\qquad
0\leq j\leq n-1;
\end{equation}
see \cite[Prop.~2.2]{HuiskenSinestrari99}.
In particular,
\begin{equation}
E_{j;i}(\lambda)>0
\qquad
\text{if }\lambda\in\bbR_{>0}^n,
\quad 0\leq j\leq n-1.
\label{eq:Ej-omitted-positive}
\end{equation}

We now compute an explicit expression for this operator.

\begin{lemma}
At a point of $\Sn$, choose a $\bar g$-orthonormal frame
$\{e_i\}_{i=1}^n$ diagonalizing $\tau[h]$, and denote its
eigenvalues by $\lambda_1,\ldots,\lambda_n$.
Then
\begin{equation}
\begin{split}
\cL_\alpha(h)[\psi]
&=
\cF_\alpha^{ij}(\tau[h])\tau[\psi]_{ij}
\\
&=
\frac{1}{n}
\sum_{i=1}^n
\left(
\sum_{k=1}^n k\alpha_kE_{k-1;i}(\lambda)
\right)
\tau[\psi](e_i,e_i).
\end{split}
\label{eq:Falpha-linearization}
\end{equation}
The operator $\cL_\alpha(h)$ is elliptic.
\end{lemma}

\begin{proof}
The first identity follows from the linearity of $\tau$ and
\eqref{eq:Falpha-first-derivative}.
By \eqref{eq:Ej-derivative}, in the chosen frame we have
\[
\left.\frac{d}{dt}\right|_{t=0}
E_k(\tau[h]+t\tau[\psi])
=
\frac{k}{n}
\sum_{i=1}^n
E_{k-1;i}(\lambda)\tau[\psi](e_i,e_i).
\]
Summing with coefficients $\alpha_k$ proves the second identity.
This formula also holds at repeated eigenvalues, since $E_k$
is a polynomial in the entries of the endomorphism.
Since $\tau[h]>0$, \eqref{eq:Ej-omitted-positive} implies
\[
\sum_{k=1}^n k\alpha_kE_{k-1;i}(\lambda)>0
\]
for every $i$.
Hence the principal coefficient tensor is positive definite,
and $\cL_\alpha(h)$ is elliptic.
\end{proof}

By \eqref{eq:E-D},
\begin{equation}
\cF_\alpha^{ij}(\tau[h])\tau[h]_{ij}
=
\sum_{k=1}^n
k\alpha_kE_k(\tau[h]).
\label{eq:Falpha-Euler}
\end{equation}

Since $\tau[h]$ is a Codazzi tensor, the Newton tensor associated
with each $E_k(\tau[h])$ is divergence-free; see
\cite[Lem.~18.30]{ACGL20}.
Hence, by linearity, 
\begin{equation}
\bar\nabla_i
\bigl(\cF_\alpha^{ij}(\tau[h])\bigr)
=
0.
\label{eq:Falpha-divergence-free}
\end{equation}

\subsection{A local generalized Alexandrov--Fenchel inequality}

The polynomial associated with the linearization of
$\cF_\alpha$ is obtained by restricting $s_\alpha$ to the
diagonal $x=y$:
\begin{equation}
q_\alpha(x,z)
:=
s_\alpha(x,x,z)
=
\sum_{j=1}^n
j\alpha_jx^{j-1}z^{n-j}.
\end{equation}
By Proposition~\ref{prop:secant-DL} and the preservation of dual
Lorentzianity under non-negative linear substitutions
\cite[Thm.~5.12]{RSW}, $q_\alpha$ is a non-zero dually Lorentzian
polynomial of degree $n-1$.

For $u_0,\ldots,u_n\in C^\infty(\Sn)$, extend the mixed-volume
notation by setting
\begin{equation}
V(u_0,\ldots,u_n)
:=
\frac{1}{n+1}
\int_{\Sn}
u_0D(\tau[u_1],\ldots,\tau[u_n])
\,d\sigma.
\label{eq:mixed-volume-functions}
\end{equation}
Every smooth function $u$ is a difference of smooth support
functions with positive definite $\tau$: for sufficiently large
$R>0$, both $u+R$ and $R$ have this property.
Thus \eqref{eq:mixed-volume-functions} is the multilinear extension
of the usual mixed volume and is symmetric in all its arguments.
For $u,v\in C^\infty(\Sn)$, define
\begin{equation}
\begin{split}
\mathcal V_{\alpha,h}(u,v)
&:=
V(u,v,q_\alpha(h,1))
\\
&=
\sum_{k=1}^n
k\alpha_k
V(u,v,h[k-1],1[n-k]).
\end{split}
\end{equation}
By \eqref{eq:Falpha-first-derivative},
\begin{equation}
(n+1)\mathcal V_{\alpha,h}(u,v)
=
\int_{\Sn}
u\,\cF_\alpha^{ij}(\tau[h])\tau[v]_{ij}
\,d\sigma.
\label{eq:Valpha-integral}
\end{equation}
The symmetry of $\mathcal V_{\alpha,h}$ shows that
$\cL_\alpha(h)$ is formally self-adjoint with respect to $d\sigma$.
In particular,
\[
\mathcal V_{\alpha,h}(h,h)>0,
\]
since it is a non-zero non-negative linear combination of mixed
volumes of bodies with non-empty interior.

The following is a specialization of the Hodge--Riemann
bilinear relations in \cite[Thm.~8.9(2)]{RSW}, expressed in
terms of smooth functions as in Eq.~(7) of its proof.
For completeness, we give a direct proof based on
Theorem~\ref{thm:RSW}, following the argument in
\cite[Prop.~18.35]{ACGL20}.

\begin{lemma}
\label{lem:local-generalized-AF}
For every $u\in C^\infty(\Sn)$,
\begin{equation}
\mathcal V_{\alpha,h}(u,h)^2
\geq
\mathcal V_{\alpha,h}(u,u)
\mathcal V_{\alpha,h}(h,h).
\label{eq:generalized-AF-linearized}
\end{equation}
Equality holds if and only if
\begin{equation}
u=ch+\ell
\end{equation}
for some $c\in\bbR$ and some linear function
$\ell(x)=\langle x,v\rangle$, $v\in\bbR^{n+1}$.
\end{lemma}

\begin{proof}
Since $\tau[h]>0$, for sufficiently large $t>0$,
\[
\tau[u+th]=\tau[u]+t\tau[h]>0.
\]
Hence $u+th$ is the support function of a smooth convex body
with positive principal curvatures.

Since $q_\alpha$ is dually Lorentzian of degree
$n-1=(n+1)-2$, Theorem~\ref{thm:RSW} gives
\[
\mathcal V_{\alpha,h}(u+th,h)^2
\geq
\mathcal V_{\alpha,h}(u+th,u+th)
\mathcal V_{\alpha,h}(h,h).
\]
By bilinearity,
\begin{equation}
\begin{split}
&\mathcal V_{\alpha,h}(u+th,h)^2
-
\mathcal V_{\alpha,h}(u+th,u+th)
\mathcal V_{\alpha,h}(h,h)
\\
&\qquad=
\mathcal V_{\alpha,h}(u,h)^2
-
\mathcal V_{\alpha,h}(u,u)
\mathcal V_{\alpha,h}(h,h).
\end{split}
\end{equation}
This proves \eqref{eq:generalized-AF-linearized}.

If equality holds, then by Theorem~\ref{thm:RSW},
\[
u+th=ah+\ell
\]
for some $a>0$ and some linear function $\ell$.
Thus $u=(a-t)h+\ell$.

Conversely, if $u=ch+\ell$, then $\tau[\ell]=0$.
By symmetry and \eqref{eq:Valpha-integral},
$\mathcal V_{\alpha,h}(\ell,v)=0$ for every smooth $v$.
Consequently,
\[
\mathcal V_{\alpha,h}(u,h)
=
c\mathcal V_{\alpha,h}(h,h),
\qquad
\mathcal V_{\alpha,h}(u,u)
=
c^2\mathcal V_{\alpha,h}(h,h),
\]
and therefore equality holds.
\end{proof}

\subsection{A weighted spectral inequality}

We now obtain a weighted extension of the spectral inequality
in \cite[Lem.~3.1]{IM23}, including its equality case.
The proof follows the same integration-by-parts argument,
with the local Alexandrov--Fenchel inequality replaced by
Lemma~\ref{lem:local-generalized-AF}.

\begin{lemma}
\label{lem:weighted-spectral-inequality}
Let $h\in C^\infty(\Sn)$ be the support function of a smooth
convex body with positive principal curvatures and containing
the origin in its interior.
If $f\in C^\infty(\Sn)$ satisfies
\begin{equation}
\int_{\Sn}
fh
\left(
\sum_{k=1}^n k\alpha_kE_k(\tau[h])
\right)
\,d\sigma
=
0,
\label{eq:weighted-orthogonality}
\end{equation}
then
\begin{equation}
\begin{split}
&\int_{\Sn}
f^2h
\left(
\sum_{k=1}^n k\alpha_kE_k(\tau[h])
\right)
\,d\sigma
\leq
\int_{\Sn}
h^2\cF_\alpha^{ij}(\tau[h])
\bar\nabla_i f\,\bar\nabla_j f
\,d\sigma.
\end{split}
\label{eq:weighted-spectral-inequality}
\end{equation}
Equality holds if and only if there exists $v\in\bbR^{n+1}$
such that
\begin{equation}
f(x)=\frac{\langle x,v\rangle}{h(x)}
\qquad
\text{for every }x\in\Sn.
\end{equation}
\end{lemma}

\begin{proof}
By \eqref{eq:Valpha-integral} and \eqref{eq:Falpha-Euler},
the orthogonality condition implies
\begin{equation}
\begin{split}
(n+1)\mathcal V_{\alpha,h}(fh,h)
&=
\int_{\Sn}
fh\sum_{k=1}^n k\alpha_kE_k(\tau[h])
\,d\sigma
=0.
\end{split}
\label{eq:Valpha-orthogonality}
\end{equation}
Since $\mathcal V_{\alpha,h}(h,h)>0$,
Lemma~\ref{lem:local-generalized-AF} implies
\begin{equation}
\mathcal V_{\alpha,h}(fh,fh)\leq0.
\label{eq:Valpha-nonpositive}
\end{equation}

By the product rule,
\begin{equation}
\begin{split}
\tau[fh]_{ij}
={}
f\tau[h]_{ij}
&+
h\bar\nabla_i\bar\nabla_j f
\\
&+
\bar\nabla_i f\,\bar\nabla_j h
+
\bar\nabla_j f\,\bar\nabla_i h.
\end{split}
\end{equation}
Therefore, by \eqref{eq:Falpha-Euler},
\begin{equation}
\begin{split}
fh\,\cF_\alpha^{ij}(\tau[h])\tau[fh]_{ij}
={}&
f^2h\sum_{k=1}^n k\alpha_kE_k(\tau[h])
+
fh^2\cF_\alpha^{ij}(\tau[h])
\bar\nabla_i\bar\nabla_j f
\\
&+
2fh\,\cF_\alpha^{ij}(\tau[h])
\bar\nabla_i f\,\bar\nabla_j h.
\end{split}
\end{equation}
Using \eqref{eq:Falpha-divergence-free}, we integrate by parts.
The mixed terms cancel, and we are left with
\begin{equation}
\begin{split}
(n+1)\mathcal V_{\alpha,h}(fh,fh)
={}&
\int_{\Sn}
f^2h\sum_{k=1}^n k\alpha_kE_k(\tau[h])
\,d\sigma
\\
&-
\int_{\Sn}
h^2\cF_\alpha^{ij}(\tau[h])
\bar\nabla_i f\,\bar\nabla_j f
\,d\sigma.
\end{split}
\label{eq:weighted-energy-identity}
\end{equation}
Combining this with \eqref{eq:Valpha-nonpositive} proves
\eqref{eq:weighted-spectral-inequality}.

If equality holds, then
$\mathcal V_{\alpha,h}(fh,fh)=0$.
Together with \eqref{eq:Valpha-orthogonality}, this is equality
in \eqref{eq:generalized-AF-linearized}.
By Lemma~\ref{lem:local-generalized-AF}, we therefore have
\[
fh=ch+\ell
\]
for some constant $c$ and some linear function $\ell$.
Since $\mathcal V_{\alpha,h}(\ell,h)=0$, we have
\[
0
=
\mathcal V_{\alpha,h}(fh,h)
=
c\mathcal V_{\alpha,h}(h,h),
\]
so $c=0$. Thus $f=\ell/h$.
Conversely, if $f=\ell/h$ with $\ell$ linear, then
$\tau[fh]=\tau[\ell]=0$.
By symmetry and \eqref{eq:Valpha-integral},
\[
\mathcal V_{\alpha,h}(fh,h)=0,
\qquad
\mathcal V_{\alpha,h}(fh,fh)=0.
\]
The first identity is the orthogonality condition.
By the second and \eqref{eq:weighted-energy-identity},
equality holds in \eqref{eq:weighted-spectral-inequality}.
\end{proof}

\begin{rem}
When only one coefficient $\alpha_k$ is non-zero,
Lemma~\ref{lem:weighted-spectral-inequality} reduces to
\cite[Lem.~3.1]{IM23}, after accounting for the normalization of
$E_k$.
For several non-zero coefficients, the weighted orthogonality
condition \eqref{eq:weighted-orthogonality} does not imply
orthogonality with respect to each measure
$hE_k(\tau[h])\,d\sigma$ separately.
Thus the result does not follow by simply adding the individual
spectral inequalities.
\end{rem}

\subsection{The kernel of the linearized operator}

\begin{prop}
For every closed smooth convex hypersurface $\cM$ with positive
principal curvatures,
\begin{equation}
\ker\bigl(\cL_\alpha(h_{\cM})\bigr)
=
\left\{
\langle\cdot,v\rangle:
v\in\bbR^{n+1}
\right\}
=:
\mathcal H_1.
\end{equation}
\end{prop}

\begin{proof}
Write $h:=h_{\cM}$.
For every linear function $\ell(x)=\langle x,v\rangle$,
$\tau[\ell]=0$, so \eqref{eq:Falpha-linearization} gives
$\cL_\alpha(h)[\ell]=0$.
Thus
\[
\mathcal H_1\subseteq\ker\bigl(\cL_\alpha(h)\bigr).
\]

Conversely, let $\psi\in\ker(\cL_\alpha(h))$.
By \eqref{eq:Valpha-integral}, for every $u\in C^\infty(\Sn)$,
\begin{equation}
(n+1)\mathcal V_{\alpha,h}(u,\psi)
=
\int_{\Sn}
u\,\cL_\alpha(h)[\psi]\,d\sigma
=
0.
\end{equation}
Taking $u=h$ and $u=\psi$ and using symmetry, we then have
\[
\mathcal V_{\alpha,h}(\psi,h)=0,
\qquad
\mathcal V_{\alpha,h}(\psi,\psi)=0.
\]
Equality holds in \eqref{eq:generalized-AF-linearized}, so
Lemma~\ref{lem:local-generalized-AF} implies
\[
\psi=ch+\ell
\]
for some $c\in\bbR$ and some linear function $\ell$.
Since $\cL_\alpha(h)[\ell]=0$,
\[
0
=
\cL_\alpha(h)[\psi]
=
c\,\cL_\alpha(h)[h].
\]
By \eqref{eq:Falpha-Euler},
\[
\cL_\alpha(h)[h]
=
\sum_{k=1}^n k\alpha_kE_k(\tau[h])
>
0.
\]
Therefore $c=0$, so $\psi=\ell\in\mathcal H_1$.
\end{proof}

\section*{Acknowledgments}

The author would like to thank M. N. Ivaki for valuable discussions
and helpful comments during the preparation of this paper.
The author was supported by the Austrian Science Fund (FWF)
under Project P36545.
The author prepared this manuscript with the aid of AI.


\begin{thebibliography}{99}

\bibitem{ACGL20}
B.~Andrews, B.~Chow, C.~Guenther, M.~Langford,
\emph{Extrinsic Geometric Flows},
Graduate Studies in Mathematics, vol.~206,
American Mathematical Society, Providence, RI, 2020.

\bibitem{BrandenHuh}
P.~Br\"and\'en, J.~Huh,
\emph{Lorentzian polynomials},
Ann. of Math. (2) \textbf{192} (2020), no.~3, 821--891.

\bibitem{Chern}
S.-S.~Chern,
\emph{Integral formulas for hypersurfaces in Euclidean space
and their applications to uniqueness theorems},
J. Math. Mech. \textbf{8} (1959), 947--955.

\bibitem{Firey74}
W.~J.~Firey,
\emph{Kinematic measures for sets of support figures},
Mathematika \textbf{21} (1974), 270--281.

\bibitem{FireyICM74}
W.~J.~Firey,
\emph{Some open questions on convex surfaces},
Proceedings of the International Congress of Mathematicians
(Vancouver, B.C., 1974), Vol.~1,
Canadian Mathematical Congress, Montreal, 1975, 479--484.

\bibitem{HuiskenSinestrari99}
G.~Huisken, C.~Sinestrari,
\emph{Convexity estimates for mean curvature flow and singularities
of mean convex surfaces},
Acta Math. \textbf{183} (1999), 45--70.

\bibitem{Ivaki23}
M.~N.~Ivaki,
\emph{Uniqueness of solutions to a class of non-homogeneous
curvature problems},
arXiv:2307.06252, 2023.

\bibitem{IM23}
M.~N.~Ivaki and E.~Milman,
\emph{Uniqueness of solutions to a class of isotropic curvature problems},
Adv. Math. \textbf{435} (2023), Art.~109350.

\bibitem{RSW}
J.~Ross, H.~S\"u{\ss}, T.~Wannerer,
\emph{Dually Lorentzian polynomials},
Monatsh. Math. \textbf{208} (2025), 495--524.

\bibitem{Schneider75}
R.~Schneider,
\emph{Kinematische Ber\"uhrma{\ss}e f\"ur konvexe K\"orper},
Abh. Math. Sem. Univ. Hamburg \textbf{44} (1975), 12--23.

\bibitem{Schneider76}
R.~Schneider,
\emph{Bestimmung eines konvexen K\"orpers durch gewisse
Ber\"uhrma{\ss}e},
Arch. Math. (Basel) \textbf{27} (1976), 99--105.

\bibitem{Schneider14}
R.~Schneider,
\emph{Convex Bodies: The Brunn--Minkowski Theory},
2nd expanded ed.,
Encyclopedia of Mathematics and its Applications, vol.~151,
Cambridge University Press, Cambridge, 2014.

\end{thebibliography}
\end{document}